\documentclass[11pt]{amsart}

\usepackage[paperheight=11in, 
    paperwidth=8.5in,
    outer=1in,
    inner=1in,
    bottom=1in,
    top=1in]{geometry}

\usepackage[all]{xy}
\usepackage{amsmath, amscd, amsthm, amsfonts, amssymb, mathrsfs}
\usepackage{graphicx}
\usepackage{url}
\usepackage{color}
\usepackage{bbm}
\usepackage{todonotes}
\usepackage[sc]{mathpazo}
\usepackage[T1]{fontenc}

\usepackage[hidelinks]{hyperref} 
 \hypersetup{
    colorlinks,
    citecolor=magenta,
    filecolor=magenta,
    linkcolor=blue,
    urlcolor=black
}
\usepackage{cleveref}

\newcommand{\RR}{\mathbb R}

\newcommand{\ZZ}{\mathbb Z}
\newcommand{\NN}{\mathbb N}
\newcommand{\FF}{\mathbb F}

\renewcommand{\P}{\mathcal{P}}
\newcommand{\C}{\mathtt{C}}
\newcommand{\D}{\mathtt{D}}

\newcommand{\hdim}{\dim_\mathtt{h}}
\newcommand{\vcdim}{\dim_\textup{VC}}
\newcommand{\rank}{\operatorname{rank}}
\newcommand{\pdim}{\operatorname{projdim}}

\DeclareMathOperator{\reg}{reg}

\theoremstyle{definition}
\newtheorem{thm}{Theorem}[section]
\newtheorem{cor}[thm]{Corollary}
\newtheorem{lem}[thm]{Lemma}
\newtheorem{prop}[thm]{Proposition}
\newtheorem{defn}[thm]{Definition}
\newtheorem{eg}[thm]{Example}
\newtheorem{rem}[thm]{Remark}

\usepackage{thm-restate}

\newcommand{\VertSet}{\mathrm{Vert}}

\newcommand{\sbe}{\subseteq}

\title{Free resolutions of function classes via order complexes}
\author{Justin Chen, Christopher Eur, Greg Yang, Mengyuan Zhang}

\address{Georgia Institute of Technology. Atlanta, GA. USA}
\email{justin.chen@math.gatech.edu}

\address{University of California, Berkeley. Berkeley, CA. USA.}
\email{ceur@math.berkeley.edu}

\address{Microsoft Research AI. Redmond, WA. USA.}
\email{gregyang@microsoft.com}

\address{University of California, Berkeley. Berkeley, CA. USA.}
\email{myzhang@berkeley.edu}

\begin{document}

\maketitle

\begin{abstract}
Function classes are collections of Boolean functions on a finite set, which are fundamental objects of study in theoretical computer science.  We study algebraic properties of ideals associated to function classes previously defined by the third author.  We consider the broad family of intersection-closed function classes, and describe cellular free resolutions of their ideals by order complexes of the associated posets.  For function classes arising from matroids, polyhedral cell complexes, and more generally interval Cohen-Macaulay posets, we show that the multigraded Betti numbers are pure, and are given combinatorially by the M\"obius functions.  We then apply our methods to derive bounds on the VC dimension of some important families of function classes in learning theory.
\end{abstract}

\section{Introduction}

For $n \in \NN$, let $[n] := \{0, 1, \ldots, n-1\}$.  A \emph{function class}\footnote{In learning theory, such Boolean function classes are also called \emph{concept classes}.} $\C$ is a collection of Boolean functions on $[n]$, that is, $\C \subseteq [2]^{[n]}$. 

A central question of learning theory is:
\begin{center}
    \it
    how much data is required to learn an unknown function $f^*$,

    given that $f^*$ is in some known function class $\C$?
\end{center}
Here, \emph{to learn $f^*$} means to identify some function $\hat f\in \C$ such that $\hat f$ is identical to $f^*$ except on a small subset\footnote{More precisely, by a small subset we mean that the subset has small probability under the distribution that the data is drawn from.} of $[n]$. 

The classical answer to the question above is given by the VC dimension.

\begin{defn} \label{defn:VCdim}
We say a subset $U \subseteq [n]$ is \textit{shattered by} $\C$ if every function on $U$ is a restriction of some function in $\C$.
The \textit{VC dimension} (Vapnik-Chervonenkis dimension) \cite{vc} of $\C$ is 
$$\vcdim \C := \max \{ |U| \, \big| \, U \text{ is shattered by } \C \}.$$
\end{defn}

For more than 40 years since its introduction, the VC dimension has occupied center-stage in learning theory and other \emph{analytically-flavored} branches of computer science.
It is a celebrated theorem in classical learning theory that the number of samples needed to learn an unknown function in $\C$ is proportional to $\vcdim \C$; we point to \cite{kearns_introduction_1994} for precise statements and more details on learning theory.

\medskip
In this paper, we continue the study of the learning theoretic properties of $\C$ using invariants of \emph{homological nature} introduced by the third author \cite{Yan17}.
There is a natural simplicial complex $\Diamond_\C$ associated to a function class $\C$, called the \emph{suboplex} of $\C$. We consider the Stanley-Reisner ideal $I_\C$ of $\Diamond_\C$ as well as its dual ideal $I^\star_\C$ --- see \S\ref{section:suboplexes} for details.  One can then analyze the learning theoretic properties of $\C$ by drawing upon the vast literature on squarefree monomial ideals.  

\begin{thm}\cite[Theorem 3.11]{Yan17}
Define the \textit{homological dimension} $\hdim \C$ of a function class $\C$ as the projective dimension of $I^\star_\C$, i.e. $\hdim \C := \pdim I^\star_\C$.  Then
$$\vcdim \C \le \hdim \C.$$
\end{thm}

The two quantities $\vcdim$ and $\hdim$ can be different, but they do coincide for many function classes of importance in computer science \cite[Section 3.1]{Yan17}, such as the class of parity functions, the class of polynomial threshold functions, or the class of monotone conjunctions.  Our goal in this paper is two-fold: (i) to investigate the multigraded Betti numbers of $I^\star_\C$, and (ii) to identify new large families of function classes for which $\vcdim$ and $\hdim$ coincide or approximately coincide.

Our main cases of interest are function classes with suitable semi-lattice structures. 
Let $\P$ be a subposet of the lattice of subsets of $[n]$ that is intersection-closed (see \S\ref{section:semilattice}). We consider the function class $\C(\P)$ defined by $\P$ by identifying subsets of $[n]$ with their indicator functions as in \cite{HSW89}. Function classes arising in this way include:
\begin{itemize}
    \item conjunctions (logical AND) of parity functions (i.e.\ conjunctions of linear functionals over $\FF_2$), and more generally the lattice of flats of a matroid (see \S\ref{subsection:matroids} and \S\ref{subsection:comp2}),
    \item downward-closed classes, and more generally the face poset of a polyhedral cell complex (see \S\ref{subsection:facePosetCellComplex}), and
    \item the class of $k$-CNFs (conjunctive normal forms) and the class of $CSP$s (constraint satisfaction problems) (see \S\ref{subsection:comp1}).
\end{itemize}

Our main result is a construction of an explicit free resolution of the ideal $I^\star_{\C(\P)}$ via the order complex $\Delta_{\P}$ of $\P$.  We refer to \S\ref{section:posets} and \S\ref{section:semilattice} for the relevant definitions and notation.

\newtheorem*{thm:resolution}{\Cref{thm:resolution}}
\begin{thm:resolution}
Let $\P \subseteq 2^{[n]}$ be an intersection-closed poset, and $S$ the polynomial ring of the ideal $I^\star_{\C(\P)}$.
Denote by $\mathcal F(\Delta_\P)$ the cellular chain complex of free $S$-modules arising from the order complex $\Delta_\P$ of $\P$, with vertices labelled by minimal generators of $I^\star_{\C(\P)}$.
Then $\mathcal F(\Delta_\P)$ is acyclic and hence gives an $S$-free resolution of $S/I_{\C(\P)}^\star$.
\end{thm:resolution}

As a consequence, we obtain the multigraded Betti numbers of $I^\star_{\C(\P)}$ purely in terms of the poset topology of $\P$.

\theoremstyle{definition}
\newtheorem*{thm:Betti1}{\Cref{thm:Betti1}}
\begin{thm:Betti1}
Let $\P \subseteq 2^{[n]}$ be an intersection-closed poset.  Then the nonzero multigraded Betti numbers of $I^\star_{\C(\P)}$ occur precisely in the degrees of certain monomials denoted $m(A,B)$ associated to each $A\leq B \in \P$ (see \Cref{defn:dictionary}), and in these degrees,
\[
\beta_{i,m(A,B)}\left(I_{\C(\P)}^\star\right) = \dim_{\mathbbm k} \widetilde{H}_{i-2}(\overline{\Delta}_{[A,B]};\mathbbm k), \quad \forall i\ge 1
\]
where $\overline{\Delta}_{[A,B]}$ denotes the truncated order complex of the interval $[A,B]$ (Definition \ref{defn:trunc}).
\end{thm:Betti1}

\theoremstyle{definition}
\newtheorem*{cor:pdim<rank}{\Cref{cor:pdim<rank}}
\begin{cor:pdim<rank}
Let $\P \subseteq 2^{[n]}$ be an intersection-closed poset. Then $\hdim \C(\P) \leq \rank(\P)$.
\end{cor:pdim<rank}

If furthermore all (open) intervals of $\P$ are Cohen-Macaulay (in which case we say $\P$ is \emph{interval Cohen-Macaulay}),
as is the case when $\P$ arises from a matroid or a polyhedral cell complex, then the multigraded Betti numbers of $I^\star_{\C(\P)}$ are pure and can be described combinatorially by the M\"obius function of $\P$ (see \S\ref{section:locCM}).

\newtheorem*{thm:Betti2}{\Cref{thm:Betti2}}
\begin{thm:Betti2}
Let $\P \subseteq 2^{[n]}$ be an intersection-closed poset that is interval Cohen-Macaulay, and $\mu_{\mathcal P}(\cdot, \cdot)$ the M\"obius function of the poset $\mathcal  P$.  Then (with $m(A, B)$ as in Theorem \ref{thm:Betti1}),
$$\beta_{i,m(A,B)}\left(I^\star_{\C(\P)}\right) = 
\begin{cases} |\mu_\P(A,B)| & \textnormal{if }\rank([A,B]) = i\\
0 & \textnormal{otherwise.}
\end{cases}
$$
\end{thm:Betti2}

For example, when $\P = \P_M$ is the lattice of flats of a matroid $M$, the quantity $\mu_\P(F,G)$ for $F \subseteq G \in \P$ is known as the \emph{M\"obius number} of the matroid minor $M|G/F$.  When $\P = \P_X$ is the face poset of a polyhedral cell complex $X$, then the quantity $\mu_\P(F,G)$ for $F \subseteq G \in \P$ is the reduced Euler characteristic of the boundary complex of a polytope, which is always $\pm 1$.  

\medskip
In section \S\ref{section:comp}, we apply our tools to give bounds for the VC dimension of various function classes of importance in learning theory, such as the class of $k$-CNFs, the class of $\D$-CSPs, the class of conjunctions of parity functions, and more generally the class of conjunctions of polynomials over $\mathbb F_2$ (see \S\ref{section:comp} for definitions of these classes).

\newtheorem*{thm:kCNF}{\Cref{thm:kCNF}}
\newtheorem*{thm:F2conj}{\Cref{cor:parityconj} \& \ref{cor:polyconj}}

\begin{thm:kCNF}
Let $\C$ be the class of $k$-CNFs and let $\C^+$ be the class of monotone $k$-CNFs in $d$ variables.
Then
\begin{align*}
    \Omega(d^k) &\le \vcdim \C^+ \le \hdim \C^+ \le O(d^k)\\
    \Omega(d^k) &\le \vcdim \C \le \hdim \C \le O(d^k),
\end{align*}
where $\Omega$ and $O$ hides constants dependent on $k$ but independent of $d$.

\end{thm:kCNF}

\begin{thm:F2conj}
Homological dimension and VC dimension coincide for the class of conjunctions of parity functions.
The same holds more generally for conjunctions of degree-bounded polynomials over $\FF_2$.
\end{thm:F2conj}

\subsection*{Acknowledgements} The first, second, and the fourth authors acknowledge the support of Microsoft Research, where the third author hosted visits which led to this work. The third author also thanks Chris Meek for discussions and advice during the writing of this paper.

\section{Suboplexes and shatter complexes} \label{section:suboplexes}

Let $\C \subseteq [2]^{[n]}$ be a class of Boolean functions on $[n]$. Following \cite{Yan17}, we define a simplicial complex from $\C$ as follows:

\begin{defn} \label{simpCplxDef}
The \textit{suboplex} $\Diamond_\C$ associated to $\C$ is the simplicial complex with vertex set $[n] \times [2]$ corresponding to input-output pairs $(i, b)$ with $i \in [n]$, $b \in [2]$, and has facets given by graphs of functions in $\C$, i.e.\ $\text{facets}(\Diamond_\C) = \{ \{(i, f(i)) \mid i \in [n]\} \mid f \in \C \}$.
\end{defn}

\begin{rem} \label{rem:allVertices}
It is possible that some points in $[n] \times [2]$ do not appear in the graph of any function of $\C$.  For example, if $f(i) = 0$ for every $f\in\C$ for some $i \in [n]$, then no facet of $\Diamond_\C$ contains the vertex $(i,1)$ . In this case, $\Diamond_\C$ is the cone over the suboplex $\Diamond_{\C|_{[n]\setminus \{i\}}}$ of the restriction of $\C$ to $[n]\setminus \{i\}$, and passing to $[n] \setminus \{i\}$ loses no learning theoretic information about $\C$.
\end{rem}

We now define the main algebraic object of study:

\begin{defn} \label{SRidealDef}
The \textit{suboplex ideal} $I_\C$ is the Stanley-Reisner ideal associated to the simplicial complex $\Diamond_\C$.  Explicitly, in the polynomial ring $S := \mathbbm k[x_{(i,b)} \mid (i, b) \in [n] \times [2]]$ over a fixed field $\mathbbm k$, $I_\C$ is a squarefree monomial ideal whose monomial minimal generators are the minimal nonfaces of $\Diamond_\C$. We will often consider the Alexander dual of the suboplex ideal, 
i.e. $\displaystyle I^\star_\C := \langle \prod_{i \in [n]} x_{(i, 1 - f(i))} \mid f \in \C \rangle$.
\end{defn}

A \emph{partial function} on $[n]$ is a function $f: A \to [2]$ defined on some subset $A \subseteq [n]$, and we denote by $\operatorname{dom}(f) = A$ its domain.  Henceforth, the term ``function'' (without the modifier ``partial'') will always mean a complete function $[n] \to [2]$.
Given $A \subseteq B \subseteq [n]$ and partial functions $f : A \to [2]$, $g: B \to [2]$, we say that $f$ is a \emph{restriction} of $g$, or equivalently $g$ is an \emph{extension} of $f$, if $g|_A = f$.  
We can describe the monomial minimal generators of $I_\C$, which come in two types, as follows: an \textit{extenture} of $\C$ is a partial function $f$ which is not a restriction of a function in $\C$, although every proper restriction of $f$ is; and a \textit{functional monomial} is a quadratic monomial of the form $x_{(i,0)} x_{(i,1)}$ for some $i \in [n]$\footnote{Each such monomial, representing a minimal nonface of $\Diamond_\C$, encodes the fact that every function has to send $i \in [n]$ to a unique output in $[2]$, hence the name \emph{functional monomial}.}.

\begin{prop}[\cite{Yan17}, Proposition 2.38] \label{prop:mingenIC}
The ideal $I_\C$ is minimally generated by the functional monomials and monomials defined by extentures in the following way:
\[
I_\C = \Big \langle  x_{(i,0)} x_{(i,1)} \mid i \in [n] \Big\rangle + \left\langle \prod_{i \in \text{dom}(f)} x_{(i, f(i))} \; \Big| \; f \text{ extenture of } \C \right \rangle.
\]
\end{prop}

Having defined the ideals $I_\C$ and $I^\star_\C$, one can then interpret algebraic properties of $I_\C$ and $I^\star_\C$ in terms of the function class $\C$, and vice versa.  For instance:

\begin{defn}{\cite[Definition 11.6]{Yan17}} \label{defn:hdim}
The \textit{homological dimension} of $\C$, denoted $\hdim \C$, is defined to be the projective dimension of $I^\star_\C$, i.e.
\[
\hdim \C := \pdim I^\star_\C = \pdim S/I^\star_\C - 1.
\]
\end{defn}

\begin{rem} \label{hdimRegularity}
By \cite[Theorem 5.59]{MS05}, the projective dimension of a squarefree monomial ideal is related to the (Castelnuovo-Mumford) regularity of its Alexander dual in the following way:
\[
\hdim \C = \pdim(S/I^\star_\C) - 1 = \reg(I_\C) - 1 = \reg(S/I_\C).
\]
\end{rem}

Next, we interpret VC dimension algebraically.
Note that if $U$ is shattered by $\C$ (recall \Cref{defn:VCdim}), and $U' \subseteq U$, then $U'$ is also shattered by $\C$. Thus the sets shattered by $\C$ form a simplicial complex $\mathcal{SH}_\C$, called the \textit{shatter complex} of $\C$. In this way, the VC dimension of $\C$ is one more than the dimension of $\mathcal{SH}_\C$.

\begin{defn} \label{collapseMapDef}
Define the \textit{collapse map} $\pi : S \to T := \mathbbm k[y_i \mid i \in [n]]$ sending $x_{(i, b)} \mapsto y_i$. This is a surjective map, with kernel given by $\langle x_{(i,0)} - x_{(i,1)} \mid i \in [n]\rangle$, generated by linear binomials naturally corresponding to the functional monomials.
\end{defn}

The following observation is an easy consequence of \Cref{prop:mingenIC}:

\begin{prop}{\cite[Theorem 3.3]{Yan17}} \label{relatingShToDelta}
Let $I_{\mathcal{SH}_\C}$ be the Stanley-Reisner ideal of $\mathcal{SH}_\C$ in the ring $T$. Then $\pi(I_\C) = I_{\mathcal{SH}_\C} + \langle y_i^2 \mid i \in [n] \rangle$. Equivalently, $U \in \mathcal{SH}_\C$ if and only if $\prod_{i\in U}y_i \not \in \pi(I_\C)$.
\end{prop}

This yields the following algebraic description of VC dimension:

\begin{cor}
(With notation as in \Cref{collapseMapDef}) $\vcdim \C = \reg(T/\pi(I_\C))$. 
\end{cor}

\begin{proof}
Since $\langle y_i^2 \mid i \in [n] \rangle \subseteq \pi(I_\C)$, the quotient ring $T/\pi(I_\C)$ is Artinian and has a basis consisting of squarefree monomials. Hence, by \cite[Exercise 20.18]{Eis05}, we have $$\reg(T/\pi(I_\C)) = \max\{d \mid (T/\pi(I_\C))_d \ne 0\} = \max\{|U| \mid \prod_{i \in U} y_i \not\in \pi(I_\C)\}.$$
By \Cref{relatingShToDelta}, the last expression equals $\max\{|U| \mid U \in \mathcal{SH}_\C\} = \vcdim \C$.
\end{proof}

We now recall the following central result, relating the VC dimension and homological dimension of an arbitrary function class:

\begin{thm}{\cite[Theorem 3.11]{Yan17}} \label{thm:vcdim<hdim}
Let $\C \subseteq [2]^{[n]}$ be a function class. Then $\vcdim \C \le \hdim \C$.
\end{thm}

A natural question to ask is when equality in \Cref{thm:vcdim<hdim} holds. In general, the difference between $\hdim \C$ and $\vcdim \C$ can be arbitrarily large:

\begin{eg}
Consider the class $\C = \{ \delta_i \mid 1 \le i \le n \}$ of delta functions on $[n]$, where $\delta_i(j) = 1 \iff i = j$. Since the constant function $0$ is an extenture of $\C$ with domain of size $n$, the homological dimension $\hdim \C$ is at least $n-1$ (in fact, $\hdim \C = n-1$). However, $\C$ cannot shatter any subset of size $> 1$, so $\vcdim \C = 1$.
\end{eg}

This example also shows that $\vcdim \C$ cannot always be sandwiched between $\operatorname{maxdeg} I_\C - 1$ (i.e., one less than the maximal size domain of an extenture) and $\hdim \C$. 

\section{Posets and order complexes} \label{section:posets}

In light of the inequality $\vcdim \C \le \hdim \C$, the importance of determining homological invariants --- in particular a free resolution --- of $I^\star_\C$ becomes clear.
To this end, we now bring additional combinatorics into the picture, by viewing function classes as arising from posets. In doing so we lose no generality, and at the same time gain methods and viewpoints to attack our motivating question of resolving $I^\star_\C$.

Let $2^{[n]}$ be the Boolean poset of all subsets of $[n]$, partially ordered by inclusion. We consider subposets $(\P, \le)$ of $2^{[n]}$ which are compatible with the ambient Boolean poset, so that if $A \le B$ in $\P$, then $A \subseteq B$ as subsets of $[n]$. Let $\C(\P) \subseteq [2]^{[n]}$ be the function class associated to $\P$ by identifying subsets with their indicator functions: notationally, we distinguish between $2^{[n]}$ for sets and $[2]^{[n]}$ for functions.

\begin{rem}
There is a natural $(\mathbb{Z}/2\mathbb{Z})^n$ action on $[2]^{[n]}$ by flipping $0$ and $1$ in the outputs, which induces an action on the set of all function classes on $[n]$.
The learning-theoretic properties considered in this paper, most notably $\vcdim \C$, are invariant under this action.
Thus, any results for the function class $\C(\P)$ also apply to any other function class in the orbit of $\C(\P)$ under the $(\mathbb{Z}/2\mathbb{Z})^n$-action.
\end{rem}

We fix the following notation for a finite poset $\P$:
\begin{itemize}
\item By $A \le B \in \P$ we mean ``$A, B \in \P$ with $A \le B$''.
\item For $A\leq B \in \P$, we let $[A,B]$ (resp.\ $(A,B)$) denote the closed (resp.\ open) interval $$[A,B] := \{C \in \P \mid A\leq C \leq B\}, \quad (A,B) := \{C \in \mathcal  P \mid A < C < B\}.$$
\item We denote by $Ch_i(\P)$ the set of $i$-chains in $\P$, i.e.
$$Ch_i(\P) := \{C_0 < C_1 < \cdots < C_i \mid C_j \in \P \ \forall j=0,\ldots, i\}$$
and also $Ch(\P) :=\bigcup_{i\geq -1} Ch_i(\P)$, where $Ch_{-1}(\P) := \{\emptyset \}$.
\item The \emph{rank} of a poset $\P$ is $\rank(\P) := \max\{ i \ | \ Ch_i(\P) \neq \emptyset \}$.
\item The poset $\P$ is \emph{bounded} if it has a unique minimal element, denoted $\hat 0$, and a unique maximal element, denoted $\hat 1$.
\end{itemize}

\begin{defn}\label{defn:trunc}
Let $\P$ be a poset.  The \emph{order complex} $\Delta_{\P}$ \emph{associated to} $\P$ is the simplicial complex whose $i$-faces are the $i$-chains $Ch_i(\P)$.  In particular, the vertices of $\Delta_\P$ are the elements of $\P$, and the facets of $\Delta_\P$ are maximal chains in $\P$.

It is convenient to have the following variant of the order complex construction: Let $\P$ be a bounded poset. The \emph{truncated order complex} $\overline{\Delta}_{\P}$ is a simplicial complex whose $i$-faces are the $i$-chains in $Ch_i(\P)$ that neither begin with $\hat 0$ nor end with $\hat 1$.
In other words,
\[
\overline{\Delta}_\P :=
\begin{cases}
\Delta_{\P\setminus \{\hat 0, \hat 1\}} & \textnormal{if }\rank(\P) \geq 2\\
\{\emptyset \} & \textnormal{if }\rank(\P) = 1\\
\emptyset & \textnormal{if }\rank(\P) =0
\end{cases}
\]
where $\{\emptyset \}$ is the \emph{empty complex} which has a single face (namely the empty set), and $\emptyset$ is the \emph{null complex} which has no faces. By convention, for $A \le B \in \P$, we take $\Delta_{(A,B)}$ to equal $\overline{\Delta}_{[A,B]}$, even when $\rank([A,B]) = 1$ so that $(A,B)$ is the empty poset.

\end{defn}

\begin{rem}\label{rem:truncatedOrderComplex}
Two key observations are in order. First, the empty complex $\{\emptyset \}$ is distinguished by the following fact: the $(-1)$-th reduced homology of a simplicial complex is nonzero if and only if the complex is the empty complex. Second, it is easily seen that the subcomplex of $\Delta_{\P}$ consisting of chains that include $\hat 0$ or $\hat 1$, but not both, is the suspension of the truncated order complex $\overline{\Delta}_{\mathcal{P}}$.
\end{rem}

Recall that our goal is to give a free resolution, as well as Betti numbers, of the dual ideal $I^\star_{\C(\P)}$ in terms of the poset $\P$.  We prepare by fixing a convenient dictionary between monomials and (partial) functions, which will be used to describe the monomial minimal generators of $I^\star_\C$ and their least common multiples.

\begin{defn}\label{defn:dictionary}
For subsets $A \subseteq B \subseteq [n]$, define corresponding partial functions and monomials
\[
\delta(A,B)(i) := \begin{cases}
1 & \textnormal{ if }i\in A\\ 
0 & \textnormal{ if }i\not\in B
\end{cases} \qquad \longleftrightarrow \qquad m(A, B) := \prod_{i\in A} x_{(i,0)} \prod_{i\not\in B} x_{(i,1)} \prod_{i\in B\setminus A} x_{(i,0)}x_{(i,1)}
\]
where $\delta(A, B) : A \sqcup ([n] \setminus B) \to [2]$ is a partial function, and $m(A, B) \in S = \mathbbm k[x_{(i,b)} \mid (i, b) \in [n] \times [2]]$ is a monomial of degree $|A| + (n - |B|) + 2(|B \setminus A|) = n + |B \setminus A|$.
\end{defn}

\begin{rem}\label{rem:dictionaryremark}
We record some straightforward but useful observations relating subsets and partial functions:
\begin{enumerate}
    \item For any function $f: [n] \to [2]$, one has $f = \delta(A,A)$ where $A = f^{-1}(1)$.
    \item For $C\subseteq D \subseteq [n]$ and $A\subseteq B\subseteq [n]$, the partial function $\delta(C,D)$ extends $\delta(A,B)$ if and only if $A\subseteq C\subseteq D \subseteq B$.  In this case, we write $\delta(C,D) \supseteq \delta(A,B)$.
    \item For partial functions $f,g$ on $[n]$, define their intersection $f\cap g$ to be the partial function defined on $\{i \in \operatorname{dom}(f) \cap \operatorname{dom}(g) \mid f(i) = g(i)\}$ by $(f\cap g)(i) := f(i) = g(i)$.  Then for $C\subseteq D \subseteq [n]$ and $A\subseteq B\subseteq [n]$,
    $$\delta(C,D) \cap \delta(A,B) = \delta(C\cap A, D\cup B).$$
\end{enumerate}
\end{rem}

The next lemma collects more facts about the dictionary \ref{defn:dictionary} relating partial functions with monomials which will be used in the sequel; we leave the easy verifications to the reader. 

\begin{lem} \label{lem:dictionary}
Let $\P \subseteq 2^{[n]}$ be a poset, and $\C(\P) \subseteq [2]^{[n]}$ the associated function class.
\begin{enumerate}
    \item For two monomials $m, m'$, we write $m \preceq m'$ to mean that $m$ divides $m'$, and we let $\operatorname{lcm}(m,m')$ denote the least common multiple of $m$ and $m'$.  Then
    $$\delta(C,D) \supseteq \delta(A,B) \iff m(C,D) \preceq m(A,B) \textnormal{, and}$$
    $$m(C\cap A, D\cup B) = \operatorname{lcm}(m(C,D), m(A,B)).$$
    \item The ideal $I^\star_{\C(\P)}$ is minimally generated by the monomials
    \[
    I^\star_{\C(\P)} = \langle m(A,A) \mid A \in \P  \rangle.
    \]
    \item The dictionary \ref{defn:dictionary} gives an order-reversing isomorphism between the lcm-semilattice of monomial minimal generators of $I^\star_{\C(\P)}$ and the semilattice of partial functions generated by intersections of functions in $\C(\P)$. 
    \item A squarefree monomial $m \in S$ is of the form $m(A,B)$ for some $A\subseteq B\subseteq [n]$ if and only if $x_{(i,0)}$ or $x_{(i,1)}$ divides $m$, for all $i \in [n]$.
\end{enumerate}
\end{lem}
\begin{proof}
Omitted.
\end{proof}

\section{Intersection-closed function classes}\label{section:semilattice}

We now specialize to the main family of function classes under consideration.

\begin{defn} \label{defn:intclosed}
A poset $\P \subseteq 2^{[n]}$ is \emph{intersection-closed} if $A,B\in \P \implies A\cap B \in \P$. In this case, we also say the associated function class $\C(\P)$ is \emph{intersection-closed}. For any poset $\P \subseteq 2^{[n]}$ (not necessarily intersection-closed) and any subset $A\subseteq [n]$, set $V(A) := \{ B \in \P \mid A \subseteq B \}$, and define the \emph{closure} of $A$ with respect to $\P$ as $\displaystyle \overline{A} := \bigcap_{B \in V(A)} B$.
\end{defn}

A poset $\P \subseteq 2^{[n]}$ is intersection-closed if and only if: for any $A \subseteq [n]$, $V(A) \ne \emptyset \iff \overline{A} \in \P$.
As we shall see, besides forming a natural class of examples, intersection-closed function classes allow for rich interplay between algebra, combinatorics, and order theory. Intersection-closed function classes were also studied in \cite{HSW89}.
To help build intuition about these notions, we leave the proof of the following simple observation to the reader:

\begin{lem}\label{lem:posetshatter}
Let $\P \subseteq 2^{[n]}$ be an intersection-closed poset, and $U \subseteq [n]$.
Then $U$ is shattered by $\C(\P)$ if and only if $\overline{A} \cap U = A$ for all $A \subseteq U$, or equivalently $\overline{A} \cap (U \setminus A) = \emptyset$ for all $A \subseteq U$.
\end{lem}

We are now ready for our main result: an explicit free resolution of $I^\star_{\C(\P)}$ from the combinatorial data of the intersection-closed poset $\P$. To be precise, we construct a cellular free resolution of $I^\star_{\C(\P)}$ on the order complex $\Delta_\P$ as follows: label each vertex $A$ of $\Delta_\P$ by the monomial 
\[
m(A,A) = \prod_{i\in A} x_{(i,0)} \prod_{i\not\in A} x_{(i,1)}
\]
(note that under the dictionary \ref{defn:dictionary}, these monomials correspond precisely to functions $f \in \C(\P)$), and label each face of $\Delta_\P$ by the lcm of the monomials of its vertices.  Such a labeling of $\Delta_\P$ defines a complex $\mathcal{F}(\Delta_\P)$ of free $S$-modules: 

\[
\mathcal{F}(\Delta_\P): \ldots \to F_i \xrightarrow{\partial_i} F_{i-1} \xrightarrow{\partial_{i-1}} \ldots \xrightarrow{\partial_1} F_0 \xrightarrow{\partial_0} F_{-1} \to 0
\]
where $F_i = \oplus_m S(-m)^{\beta_{i,m}}$ is a free $S$-module with basis given by $i$-faces of $\Delta_\P$ and monomial shifts corresponding to labels; for details we point to \cite[Ch. 2]{Eis05} or \cite[Ch. 4]{MS05}. Note that $F_{-1} = S$, labeled by the monomial $1 \in S$ corresponding to the empty set, is the unique $(-1)$-dimensional face of $\Delta_\P$.

 \begin{thm} \label{thm:resolution}
 Let $\P \subseteq 2^{[n]}$ be an intersection-closed poset. Then $\mathcal{F}(\Delta_\P)$ is acyclic and hence gives an $S$-free resolution of $S/I_{\C(\P)}^\star$.
 \end{thm}

\begin{proof}
First, note that by \Cref{lem:dictionary}(3), the monomials appearing as a label of a face in $\mathcal F(\Delta_\P)$ are all of the form $m(A,B)$ for some $A\leq B \in \P$.
For $\mathcal F(\Delta_\P)$ to be acyclic, by \cite[Proposition 4.5]{MS05} we need to show that for any monomial $m$, the subcomplex $(\Delta_\P)_{\preceq m}$ is acyclic, where $(\Delta_\P)_{\preceq m}$ consists of all faces of $\Delta_\P$ labeled by monomials $\preceq m$.

Since all labels of $\Delta_\P$ are squarefree, it suffices to consider squarefree monomials $m$.
If $m$ is squarefree but not of the form $m(A,B)$ for some $A\subseteq B \subseteq [n]$, then by \Cref{lem:dictionary}(4) then there exists $i\in [n]$ such that neither $x_{(i,0)}$ nor $x_{(i,1)}$ divides $m$.
In particular, no monomial of the form $m(A,B)$ can divide $m$, and so in this case $(\Delta_\P)_{\preceq m} = \emptyset$ is the null complex, hence is acyclic.

We are thus left with the case where $m = m(A,B)$ for some $A, B \subseteq [n]$.
\Cref{lem:dictionary}(1) then implies that $(\Delta_\P)_{\preceq m}$ consists of chains $\{C_0 < \cdots < C_i\}$ with $A\subseteq C_0$ and $C_i \subseteq B$.
If $V(A) = \emptyset$, then $(\Delta_\P)_{\preceq m} = \emptyset$ is null. 
Thus we may assume $V(A) \ne \emptyset$, so since $\P$ is intersection-closed, $\overline{A} \in \P$. 
Observe that every vertex of $(\Delta_\P)_{\preceq m}$ is connected to the vertex $\{\overline{A}\} \in (\Delta_\P)_{\preceq m}$.
In particular, $(\Delta_\P)_{\preceq m}$ is a cone over the subcomplex of $\Delta_\P$ consisting of chains in $\P$ starting strictly above $\overline{A}$ and ending below $B$.

Putting the above reasoning together shows that $(\Delta_\P)_{\preceq m}$ is acyclic for any monomial $m$, and thus $\mathcal F(\Delta_\P)$ is a free resolution. The image of the last map $\partial_0$ in $\mathcal F(\Delta_\P)$ is the ideal in $S$ generated by the monomials $m(A,A)$ for $A\in \P$, which by \Cref{lem:dictionary}(2) is exactly $I_{\C(\P)}^\star$. 
\end{proof}

Although the resolution $\mathcal{F}(\Delta_\P)$ is in general non-minimal, we can still describe all the multigraded Betti numbers of $I_{\C(\P)}^\star$. 
The attentive reader may note that the multigraded Betti numbers of $I^\star_{\C(\P)}$ can be expressed as homologies of links of the suboplex $\Diamond_{\C(\P)}$ by Hochster's formula \cite[Corollary 1.40]{MS05}.
In view of this, the main point of the following theorem is that the Betti numbers of $I^\star_{\C(\P)}$ are expressed directly as homologies of open intervals in $\Delta_\P$, rather than links in $\Diamond_{\C(\P)}$.

\begin{thm}\label{thm:Betti1}
Let $\P \subseteq 2^{[n]}$ be an intersection-closed poset. Then the nonzero multigraded Betti numbers of $I^\star_{\C(\P)}$ can only occur in degrees $m(A,B)$ for some $A\leq B \in \P$, and
\[
\beta_{i,m(A,B)}\left(I_{\C(\P)}^\star\right) = \dim_{\mathbbm k} \widetilde{H}_{i-2}(\overline{\Delta}_{[A,B]};\mathbbm k) \quad \forall i\ge 1,
\quad \textup{ and } \beta_{0,m(A,A)}\left(I^\star_{\C(\P)}\right) = 1.
\]
\end{thm}

\begin{proof}
For $i = 0$, and any monomial $m$, we have $\beta_{0,m}(I_{\C(\P)}^\star) = 1$ iff $m = m(A, A)$ for some $A \in \P$ and $\beta_{0,m}(I_{\C(\P)}^\star) = 0$ otherwise, since $\{m(A,A)\}_{A\in \mathcal P}$ is the unique set of minimal monomial generators of $I_{\C(\P)}^\star$.  Now assume $i\geq 1$.  As in \Cref{thm:resolution}, the monomials appearing as a label of a face in $\mathcal F(\Delta_\P)$ are exactly $\{m(A,B) \mid A\leq B \in \P\}$, and since $\mathcal F(\Delta_\P)$ resolves $S/I^\star_{\C(\P)}$, the multigraded Betti numbers can only occur in such degrees.
By \cite[Theorem 4.7]{MS05}, \[
\beta_{i,{m(A,B)}}(I^\star_{\C(\P)}) = \dim_{\mathbbm k} \widetilde{H}_{i-1}( (\Delta_\P)_{\prec m(A,B)}; \mathbbm k),
\]
where $(\Delta_\P)_{\prec m(A,B)}$ is the complex whose faces are chains $\{C_0 < \ldots < C_i\}$ with $A\leq C_0$, $C_i \leq B$, and either $A \ne C_0$ or $C_i \ne B$.
By \Cref{rem:truncatedOrderComplex}, $(\Delta_\P)_{\prec m(A,B)}$ is precisely the suspension of the truncated order complex $\overline{\Delta}_{[A,B]}$, so $\widetilde{H}_{i-1}((\Delta_\P)_{\prec m(A,B)}; \mathbbm k) \cong \widetilde{H}_{i-2}(\overline{\Delta}_{[A,B]}; \mathbbm k)$.
\end{proof}

\begin{cor}
Let $\P \subseteq 2^{[n]}$ be an intersection-closed poset. Then minimal generators of first syzygies of $I^\star_{\C(\P)}$ are in bijection with cover relations in $\P$.
\end{cor}

\begin{proof}
Taking $i = 1$ in \Cref{thm:Betti1}, one has $\beta_{1,m(A,B)}(I^\star_{\C(\P)}) = \dim_{\mathbbm k} \widetilde{H}_{-1}(\overline{\Delta}_{[A,B]}; \mathbbm k)$ is nonzero $\iff \overline{\Delta}_{[A,B]} = \{ \emptyset \}$, by \Cref{rem:truncatedOrderComplex} $\iff \rank([A,B]) = 1 \iff B$ covers $A$.
\end{proof}

\begin{cor}\label{cor:pdim<rank}
Let $\P \subseteq 2^{[n]}$ be an intersection-closed poset. Then there is an inequality
\[\hdim \C(\P) \le \rank(\P),\]
and equality holds if and only if $\widetilde{H}_{rank(\P)-2} (\overline{\Delta}_{[\hat 0, B]};\mathbbm k)\ne 0$ for some maximal $B\in \P$.
\end{cor}

\begin{proof}
The inequality follows from \Cref{thm:resolution}, as the order complex $\Delta_\P$ has dimension $\rank(\P)$, so $\mathcal{F}(\Delta_\P)$ is a free resolution of $S/I^\star_{\C(\P)}$ of length $\rank(\P)+1$. 
Equality is achieved precisely when $\beta_{r,m}(I^\star_{\C(\P)}) = \widetilde{H}_{r-2} (\overline{\Delta}_{[A, B]};\mathbbm k) \neq 0$ for some monomial $m = m(A,B)$ with $r = \rank([A,B]) = \rank(\P)$, which occurs only if $B$ is maximal in $\P$ and $A = \hat 0$. \end{proof}

\section{Interval Cohen-Macaulay posets} \label{section:locCM}

Although the free resolution of $I^\star_{\C(\P)}$ for an intersection-closed poset $\P$ given in \Cref{thm:resolution} is satisfactory from an algebraic viewpoint, it is natural to ask if the Betti numbers in \Cref{thm:Betti1} have some combinatorial meaning. 
Our goal in this section is to show that for certain combinatorial families of intersection-closed posets, the multigraded Betti numbers of $I^\star_{\C(\P)}$ are given by the M\"obius function of the poset.

\begin{defn}
Let $\P$ be a poset. The \emph{M\"obius function} $\mu_\P: \P \times \P \to \ZZ$ is recursively defined by
\[
\mu_\P(A,A) := 1 \textnormal{ for any $A\in \P$ \quad and}
\]
\[\mu_\P(A,B) := - \sum_{A\leq C < B} \mu_\P(A,C) \textnormal{ for any $A\leq B \in \P$}.
\]
\end{defn}
We often drop the subscript $\P$ when the poset is clear from the context, and when $\P$ is bounded we write $\mu(\P) := \mu_\P(\hat 0, \hat 1)$.  The following statement is well-known in the literature as the Philip Hall theorem; for a proof, see e.g. \cite[Proposition 3.6]{Rot64}.

\begin{prop}\label{prop:hall}
Let $\P$ be a bounded poset with $\hat 0 \neq \hat 1$.  Then the reduced Euler characteristic of the truncated order complex $\widetilde{\chi}(\overline{\Delta}_\P)$ is given by the Mobius function, i.e. 
\[
\widetilde\chi(\overline{\Delta}_{\P}) := \sum_{i \geq -1} (-1)^i \dim_{\mathbbm k} \widetilde{H}_i (\overline{\Delta}_\P; \mathbbm k) = \mu(\P).
\]
\end{prop}

In light of \Cref{thm:Betti1}, these reduced Euler characteristics are alternating sums of multigraded Betti numbers. 
We now introduce a new variant of the Cohen-Macaulay property for posets, for which the sums in the Euler characteristics simplify to a single term.

\begin{defn}
A simplicial complex $\Delta$ is \emph{Cohen-Macaulay} (over $\mathbbm k$) if the Stanley-Reisner ideal $I_\Delta$ is Cohen-Macaulay, i.e. the quotient ring $S/I_\Delta$ is a Cohen-Macaulay ring. A poset $\P$ is \emph{Cohen-Macaulay} if the order complex $\Delta_\P$ is Cohen-Macaulay. We define a poset $\P$ to be \emph{interval Cohen-Macaulay} if every open interval in $\P$ is Cohen-Macaulay.
\end{defn}

In general, the Cohen-Macaulay property depends on the characteristic of $\mathbbm k$ (see \cite[Figure 18]{bjorner1982introduction} for an example). Next, we clarify the interval Cohen-Macaulay property.
A celebrated result of Reisner gives a characterization for a simplicial complex to be Cohen-Macaulay, in terms of homology of links. Recall that for a simplicial complex $\Delta$, the \emph{link} of a face $F$ in $\Delta$ is the subcomplex $\operatorname{lk}_\Delta F:= \{G\in \Delta \mid F\cap G = \emptyset,\ F\cup G \in \Delta \}$.

\begin{prop}\label{lemma:intervalCM} Let $\Delta$ be a simplicial complex, and $\P$ a poset.
\begin{enumerate}
    \item \cite[Theorem 1]{Rei76} $\Delta$ is Cohen-Macaulay if and only if $\widetilde{H}_i(\operatorname{lk}_\Delta F; \mathbbm k) = 0$ for all faces $F\in \Delta$ and all $i<\dim \operatorname{lk}_\Delta F$. (Note: this is stronger than saying $\widetilde{H}_i(\Delta; \mathbbm k) = 0$ for $i < \dim \Delta$.)
    \item If $\P$ is bounded with $\hat 0 \ne \hat 1$, then $\P = [\hat 0, \hat 1]$ is Cohen-Macaulay iff any of $(\hat 0, \hat 1), (\hat 0, \hat 1], [\hat 0, \hat 1)$ is Cohen-Macaulay.
    \item If $\P$ is Cohen-Macaulay, then $\P$ is interval Cohen-Macaulay.
\end{enumerate}
\end{prop}

\begin{proof}
(2) Note that $[\hat 0, \hat 1)$ is a cone over $(\hat 0, \hat 1)$ (and similarly $[\hat 0, \hat 1]$ is a cone over $[\hat 0, \hat 1)$), so any subcomplex of $[\hat 0, \hat 1)$ is either a subcomplex of $(\hat 0, \hat 1)$ or a cone over a subcomplex of $(\hat 0, \hat 1)$. Applying (1) then gives statement (2).

(3) For $A< B \in \P$, note that $\Delta_{(A,B)} = \operatorname{lk}_{\Delta_\P} C$, where $C$ is any chain in $\P$ obtained by omitting all elements strictly between $A$ and $B$ from a maximal chain in $\P$ containing $A$ and $B$.  Consequently, the links of $\Delta_{(A,B)}$ are links $\operatorname{lk}_{\Delta_\P}F$ of $\Delta_\P$ where $F$ is a chain containing $C$.
The statement of (3) now follows from (1).
\end{proof}

It follows from \Cref{lemma:intervalCM}(2) that (i) in the definition of interval Cohen-Macaulay, one could have replaced ``open'' with ``closed'', and (ii) the converse of \Cref{lemma:intervalCM}(3) is true if $\P$ is bounded. However, the converse of \Cref{lemma:intervalCM}(3) does not hold in general, as shown by the poset (left) in \Cref{fig:1} with its order complex (right).

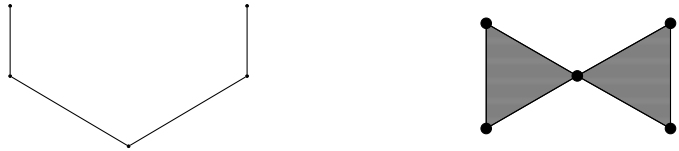
\begin{figure}[h]
    \centering
\begin{tabular}{p{3cm} p{6cm} p{6cm}}
&
\vspace{-12pt}
\begin{tikzpicture}[scale=0.7, vertices/.style={draw, fill=black, circle, inner sep=0pt}]
    \node [vertices, label=right:{}] (0) at (-0+0,0){};
    \node [vertices, label=left:{}] (1) at (-2.25+0,1.33333){};
    \node [vertices, label=right:{}] (2) at (-2.25+4.5,1.33333){};
    \node [vertices, label=left:{}] (3) at (-2.25+0,2.66667){};
    \node [vertices, label=right:{}] (4) at (-2.25+4.5,2.66667){};
    \foreach \to/\from in {0/1, 0/2, 1/3, 2/4}
    \draw [-] (\to)--(\from);
\end{tikzpicture}
&
\vspace{-4pt}
\begin{tikzpicture}[scale=0.7]
\draw[top color=gray,bottom color=gray] (0,-1) -- (1.73,0) -- (0,1) -- cycle;
\draw[top color=gray,bottom color=gray] (3.5,-1) -- (1.73,0) -- (3.5,1) -- cycle;
\draw[fill]
(0,1) circle [radius=0.1] 
(0,-1) circle [radius=0.1] 
(1.73, 0) circle [radius=0.1] 
(3.5, 1) circle [radius=0.1] 
(3.5, -1) circle [radius=0.1] 
;
\draw
(0,1) -- (0,-1) node [left] [midway] {}
(1.73, 0) -- (0,1) node [above] [midway] {}
(0,-1) -- (1.73, 0) node [below] [midway] {}
(3.5,1) -- (3.5,-1) node [right] [midway] {}
(1.73, 0) -- (3.5,1) node [above] [midway] {}
(3.5,-1) -- (1.73, 0) node [below] [midway] {}
;
\end{tikzpicture}
\end{tabular}
    \caption{A poset which is interval Cohen-Macaulay but not Cohen-Macaulay}
    \label{fig:1}
\end{figure}



\begin{thm}
\label{thm:Betti2}
Let $\P \subseteq 2^{[n]}$ be an intersection-closed poset that is interval Cohen-Macaulay.  Then the nonzero multigraded Betti numbers of $I^\star_{\C(\P)}$ occur only in degrees $m(A,B)$ for $A\leq B \in \P$, and
\[
\beta_{i,m(A,B)}(I^\star_{\C(\P)}) = 
\begin{cases} |\mu_\P(A,B)| & \textnormal{if }\rank([A,B]) = i\\
0 & \textnormal{otherwise}.
\end{cases}
\]
\end{thm}

\begin{proof}
Let $\P$ be interval Cohen-Macaulay, and $A \le B\in \P$. 
If $A = B$, then $\beta_{i,m(A,A)}(I^\star_{\C(\P)}) = 1 = \mu(A,A)$ iff $i = 0$, by \Cref{thm:Betti1}.
Otherwise, $\Delta_{(A, B)} = \overline{\Delta}_{[A,B]}$ is Cohen-Macaulay,
so by \Cref{lemma:intervalCM}(1), the summands $\dim_{\mathbbm k} \widetilde H_i(\overline{\Delta}_{[A,B]};\mathbbm k)$ in the formula of \Cref{prop:hall} are all zero except when $i = \dim \overline{\Delta}_{[A,B]} = \operatorname{rank}([A,B]) - 2$, so the result follows from \Cref{thm:Betti1}.
\end{proof}

We conclude this section by discussing two distinguished families of interval Cohen-Macaulay intersection-closed posets: (i) the lattice of flats of a matroid, and (ii) the face poset of a polyhedral complex.  In both cases, the property of being interval Cohen-Macaulay is established by shellability.

\begin{defn}
A simplicial complex on $[n]$ is \emph{shellable} if there is an ordering of its facets $F_1, \ldots, F_s$ such that $\big (\bigcup_{i < k} F_i\big) \cap F_k$ is pure of dimension $(\dim F_k - 1)$, for each $k = 2, \ldots, s$.  A poset $\P$ is \emph{shellable} if its order complex $\Delta_\P$ is shellable.
\end{defn}

The next proposition collects all results we need concerning shellability;
for details and proofs, cf. the survey \cite{Wac07} and the references therein including \cite{BM71, Sta77, Bjo80, BW96}.

\begin{prop}\label{prop:shellable}
Let $\Delta$ be a simplicial complex, and $\P$ a finite poset.
\begin{enumerate}
    \item  If $\P$ is shellable, then so is any closed or open subinterval in $\P$.
    \item If $\Delta$ is shellable, then it is Cohen-Macaulay.
    \item A locally semimodular poset, in particular a geometric lattice (i.e.\ the lattice of flats of a matroid; see \S\ref{subsection:matroids}), is shellable.
    \item The face poset of a polytope (or its boundary) is shellable.
\end{enumerate}
\end{prop}

We remark that while we only give detailed expositions for matroids and polyhedral cell complexes, there are other examples of interesting families of posets whose structure fits into the framework of this section (e.g. antimatroids, whose posets are upper-semimodular).

\subsection{Matroids}\label{subsection:matroids} 
\

\medskip
We give a brief overview of matroids  ---  details for unproven claims may be found in \cite{Oxl11}.

\begin{defn}
A \emph{matroid} $M = (E,\operatorname{rk}_M)$ consists of a finite set $E$, called the \emph{ground set}, and a \emph{rank function} $\operatorname{rk}_M: 2^E \to \ZZ_{\geq 0}$ such that
\begin{enumerate}
    \item if $A\subseteq E$ then $\operatorname{rk}_M(A) \leq |A|$,
    \item if $A\subseteq B\subseteq E$ then $\operatorname{rk}_M(A) \leq \operatorname{rk}_M(B)$, and
    \item if $A,B \subseteq E$ then $\operatorname{rk}_M(A\cup B) + \operatorname{rk}_M(A\cap B) \leq \operatorname{rk}_M(A) + \operatorname{rk}_M(B)$.
\end{enumerate}
\end{defn}

A subset $I\subseteq E$ is \emph{independent} if $\operatorname{rk}_M(I) = |I|$.
The maximal independent sets are called the \emph{bases} of $M$, and all have the same cardinality $\rank(M) := \operatorname{rk}_M(E)$, called the \emph{rank} of $M$.
The \emph{closure} of a subset $A \subseteq E$ is the set $\operatorname{cl}_M(A) := \{ x \in E \mid \operatorname{rk}_M(A) = \operatorname{rk}_M(A \cup \{x\}) \}$. 
A \emph{flat} of $M$ is a closed subset, i.e. a subset $F\subseteq E$ such that $F = \operatorname{cl}_M(F)$.  
The flats of a matroid, under inclusion, form an intersection-closed lattice $\mathcal P_M$ with $\rank(\P_M) = \rank(M)$.  
The lattice $\P_M$ is a bounded poset: there is a unique minimal flat of rank $0$, whose elements are called $\emph{loops}$, and $E$ is the unique maximal flat.
The \emph{M\"obius number} $\mu(M)$ of a matroid $M$ is the number $\mu_{\P_M}(\hat 0, \hat 1)$, where $\hat 0, \hat 1$ are respectively the bottom, top elements of $\P_M$.  
The M\"obius numbers of $\P_M$ are well-studied quantities of interest in combinatorics; see \cite{Rot64} or \cite{Zas87} for a survey.

\medskip
There are two standard ways to construct a new matroid from $M = (E, \operatorname{rk}_M)$ given a subset $T\subseteq E$: the \emph{restriction} of $M$ to $T$ is the matroid $M|T = (T, \operatorname{rk}_{M|T})$ where $\operatorname{rk}_{M|T}(A) := \operatorname{rk}_M(A)$ for $A \subseteq T$, and the \emph{contraction} of $M$ by $T$ is the matroid $M/T := (E \setminus T, \operatorname{rk}_{M/T})$ where $\operatorname{rk}_{M/T}(A) := \operatorname{rk}_M(A\cup T) - \operatorname{rk}_M(T)$ for $A \subseteq E \setminus T$.  A \emph{matroid minor} of $M$ is a matroid that arises as $M|B/A$ for some $A\subseteq B \subseteq E$.  When $F\subseteq G$ are flats of $M$, then the lattice of flats of $M|G/F$ is isomorphic to the interval $[F,G]$ in the lattice of flats of $M$.

\begin{eg} The prototypical example of a matroid is a set of vectors $E = \{v_0, \ldots, v_{n-1}\}$ in a vector space $V$, with rank function given by $\operatorname{rk}_M(A) := \dim \operatorname{span}(A)$ for $A \subseteq E$.  
In this case, the independent sets are exactly the subsets of $E$ which are linearly independent in $V$, and the flats are exactly $W \cap E$ for some vector subspace $W\subseteq V$, i.e. correspond to subspaces of $V$. Matroids arising in this way are called \emph{representable}.

When $V$ is a vector space over the finite field $\mathbb F_2$, the function class of the lattice of flats is the set of conjunctions of parity functions (i.e.\ conjunctions of
linear functionals over $\mathbb F_2$).  We discuss this case further in \S\ref{subsection:comp2}.
\end{eg}

Let $\P_M$ be the lattice of flats of a matroid $M$; lattices arising in this way are also known as \emph{geometric lattices}. It follows from \Cref{prop:shellable}(3) that $\P_M$ is shellable and hence Cohen-Macaulay.  Thus for matroids, \Cref{thm:Betti2} specializes to:

\begin{cor}\label{cor:bettimatroid}
Let $M$ be a matroid. Then for $F\leq G \in \P_M$ flats of $M$,
$$\beta_{i,m(F,G)}(I^\star_{\C(\P_M)}) = 
\begin{cases} |\mu(M|G/F)| & \textnormal{if }\rank(M|G/F) = i\\
0 & \textnormal{otherwise}.
\end{cases}
$$
In other words, the multigraded Betti numbers of $I^\star_{\C(\P_M)}$ are the M\"obius numbers of the loopless matroid minors of $M$.
\end{cor}

It is easy to check that if $M = (E, \operatorname{rk}_M)$ is a matroid, then for any subset $A \subseteq E$, the closure $\operatorname{cl}_M(A)$ of $A$ is equal to the intersection of all flats containing $A$. 
In particular, when $\P = \P_M$ is the lattice of flats of a matroid, the closure defined in \Cref{defn:intclosed} agrees with the closure operation in the matroid.
Now, if $B\subseteq E$ is a basis of $M$ and $I\subseteq B$, then the closure $\overline I$ of $I$ in $\P_M$ is disjoint from $B\setminus I$, and so it follows from \Cref{lem:posetshatter} that any basis of $M$ is shattered by $\C(\mathcal P_M)$.
Hence for any basis $B$, we have $\vcdim \C(\P_M) \geq |B| = \rank(M)$, and combining this with \Cref{cor:pdim<rank} and \Cref{thm:vcdim<hdim} yields:

\begin{cor}\label{cor:trinity1}
Let $\P_M$ be the lattice of flats of a matroid $M$, and $\C(\P_M)$ the associated function class.  Then
$$\vcdim \C(\P_M) = \hdim \C(\P_M) = \operatorname{rank}(M).$$
\end{cor}

\begin{rem}\label{rem:noCM}
A major family of function classes with the property $\vcdim \C = \hdim \C$ given in \cite{Yan17} is the case where $I_\C$ is Cohen-Macaulay.
This is true, for example, for downward-closed classes, i.e. classes $\C$ such that $g \in \C$ if $g^{-1}(1) \subseteq f^{-1}(1)$ and $f \in \C$ \cite[Section 3.2]{Yan17}.
We remark that while $\vcdim \C(\P_M) = \hdim \C(\P_M)$ for a matroid $M$, the suboplex ideal $I_{\C(\P_M)}$ is almost never Cohen-Macaulay.  Recall the Eagon-Reiner criterion \cite{ER98} that $I_\C$ is Cohen-Macaulay if and only if $I^\star_\C$ has a linear resolution.  As $\deg m(F,G) = n + |F\setminus G|$, it follows from \Cref{thm:Betti2} that $I^\star_\C$ has a linear resolution if and only if $|F\setminus G| = \rank([F,G])$ for all $F\leq G \in \P$.  Only the matroids that are Boolean after removing loops --- that is, matroids $M = (E,\operatorname{rk}_M)$ such that $\operatorname{rank}(M) = |E| - \#\textnormal{(loops)}$ --- satisfy this condition.
\end{rem}

We provide a matroidal example illustrating the theorems above.

\begin{eg}
Let $n = 4$, and let $\C$ be the function class on $[4]$ consisting of the $10$ functions
$$\{0000, 1000, 0100, 0010, 0001,1100, 1010, 1001, 0111, 1111\}.$$
Here each binary string represents a function's values on $[4]$.  Under the correspondence between subsets of $[4]$ and their indicator functions, we have $\C = \C(\mathcal P)$, where $\P = \P_M$ is the lattice of the flats of the matroid $M = U_{1,\{0\}} \oplus U_{2, \{1,2,3\}}
\cong U_{1,1}\oplus U_{2,3}$. 
Here $U_{k,m}$ is the uniform matroid of rank $k$ on $m$ elements, whose bases are all $k$-element subsets of $[m]$.
We give three different representations of $M$: a graph whose cyclic matroid is $M$, the matrix over $\mathbb F_2$ whose columns represent $M$, and the lattice of flats of $M$ are displayed in \Cref{fig:2} below.
\begin{figure}
    \centering
\begin{tabular}{p{6.3cm} p{5cm} p{5cm}}
\vspace{15pt} \hspace{0.4cm}
\begin{tikzpicture}[scale=0.7]
\draw[fill]
(0,1) circle [radius=0.1] 
(0,-1) circle [radius=0.1] 
(1.73, 0) circle [radius=0.1] 
(3.5, 0) circle [radius=0.1] 
;
\draw
(0,1) -- (0,-1) node [left] [midway] {$1$}
(1.73, 0) -- (0,1) node [above] [midway] {$3$}
(0,-1) -- (1.73, 0) node [below] [midway] {$2$}
(1.73,0) -- (3.5, 0) node [midway] [above] {0}
;
\end{tikzpicture}

&
\vspace{20pt}
$\begin{bmatrix}
1 & 0 & 0 & 0 \\
0 & 1 & 1 & 0 \\
0 & 1 & 0 & 1 \\
\end{bmatrix}
$

&
\vspace{-5pt}
\begin{tikzpicture}[scale=0.7, vertices/.style={draw, fill=black, circle, inner sep=0pt}]
    \node [vertices, label=right:{${\emptyset}$}] (0) at (-0+0,0){};
    \node [vertices, label=right:{${3}$}] (1) at (-2.25+0,1.33333){};
    \node [vertices, label=right:{${2}$}] (2) at (-2.25+1.5,1.33333){};
    \node [vertices, label=right:{${1}$}] (3) at (-2.25+3,1.33333){};
    \node [vertices, label=right:{${0}$}] (4) at (-2.25+4.5,1.33333){};
    \node [vertices, label=right:{${123}$}] (8) at (-2.25+0,2.66667){};
    \node [vertices, label=right:{${03}$}] (5) at (-2.25+1.5,2.66667){};
    \node [vertices, label=right:{${02}$}] (6) at (-2.25+3,2.66667){};
    \node [vertices, label=right:{${01}$}] (7) at (-2.25+4.5,2.66667){};
    \node [vertices, label=right:{${0123}$}] (9) at (-0+0,4){};
    \foreach \to/\from in {0/1, 0/2, 0/3, 0/4, 1/8, 1/5, 2/8, 2/6, 3/8, 3/7, 4/5, 4/6, 4/7, 5/9, 6/9, 7/9, 8/9}
    \draw [-] (\to)--(\from);
\end{tikzpicture}
\end{tabular}
\vspace{-40pt}
\caption{The matroid $U_{1,1} \oplus U_{2,3}$}
\label{fig:2}
\end{figure}
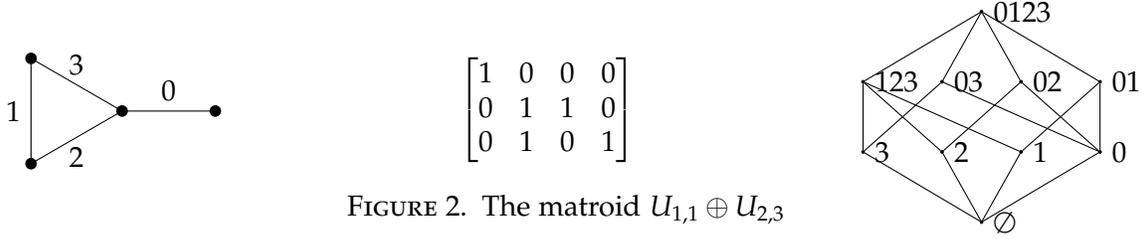

\noindent
The $\mathbb{Z}$-graded Betti table of $I^\star_\C$ is, in standard Macaulay2 \cite{M2} format,

$$\begin{matrix}
     &0&1&2&3\\\text{total:}&10&17&10&2\\\text{4:}&10&11&3&\text{.}\\\text{5:}&\text{.}&6&7&2
\end{matrix}
$$  
and in accordance with \Cref{thm:Betti2}:
\begin{itemize}
\item There are 10 flats in $\P_M$, corresponding to the $\beta_{0,4} = 10$ minimal monomial generators of $I^\star_\C$, all of degree $4$.
\item There are 17 intervals of rank 1 (i.e.\ cover relations) in $\mathcal P$, corresponding to the first total Betti number $\beta_1 = 17$. There are $\beta_{1,5} = 11$ cover relations $A < B$ with $|B \setminus A| = 1$ (so that $\deg m(A, B) = n + |B \setminus A| = 5$), and $\beta_{1,6} = 6$ cover relations with $|B \setminus A| = 2$.
\item There are 8 intervals of rank 2 in $\mathcal P$, all of which are Boolean except two: the (isomorphic) intervals $[\emptyset,123]$ and $[0,0123]$.  
The M\"obius number of a Boolean lattice is 1, and the M\"obius number of $[\emptyset,123]$ is 2, hence $\beta_2 = 6(1) + 2(2) = 10$.
As with $\beta_1$, the graded Betti numbers can be obtained as follows:  $\beta_{2,6} = 3$ is the sum of the M\"obius numbers of rank two intervals $[A, B]$ with $|B \setminus A| = 2$, and $\beta_{2,7} = 7$ is the corresponding sum with $|B \setminus A| = 3$.
\item The top Betti number is $\beta_3 = 2$, as can be verified by computing the M\"obius number of $M$. As the M\"obius number of $M$ is also the reduced Euler characteristic of the truncated order complex of $\P_M$ (cf. \Cref{defn:trunc} and \Cref{prop:hall}), we can also verify that $\beta_3 = 2$ by noting that this complex, drawn in \Cref{fig:3} below, is connected and has two-dimensional first (reduced, singular) homology.
\end{itemize}

\begin{figure}[h]
\begin{tikzpicture}[scale = 0.8]
\draw[fill]
(0,0) circle [radius = 0.07] node [left] {123}
(1,1) circle [radius = 0.07] node [left] {1}
(1,0) circle [radius = 0.07] node [above] {2}
(1,-1) circle [radius = 0.07] node [left] {3}
(2,1) circle [radius = 0.07] node [right] {01}
(2,0) circle [radius = 0.07] node [above] {02}
(2,-1) circle [radius = 0.07] node [right] {03}
(3,0) circle [radius = 0.07] node [right] {0}
;
\draw
(0,0) -- (1,1) 
(0,0) -- (1,0)
(0,0) -- (1,-1)
(1,1) -- (2,1)
(1,0) -- (2,0)
(1,-1) -- (2,-1)
(3,0) --(2,1)
(3,0) -- (2,0)
(3,0) -- (2,-1)
;
\end{tikzpicture}
\caption{The truncated order complex $\overline{\Delta}_{\P_M}$}
\label{fig:3}
\end{figure}

\noindent As per \Cref{cor:trinity1}, the homological dimension of $\C$ is 3, which is also the rank of the matroid $M$, as well as the VC dimension of $\C$; indeed, $\{0,1,2\}$ is shattered by $\C$, while the whole set $\{0,1,2,3\}$ is not.
\end{eg}

\subsection{Polyhedral cell complexes}
\label{subsection:facePosetCellComplex}

\begin{defn}\label{defn:polyhedralCellComplex}
A \emph{(polyhedral) cell complex} $X$ is a finite collection of convex polytopes (all living in a real vector space $\RR^n$), called faces of $X$, satisfying two properties:
\begin{itemize}
    \item If $F$ is a polytope in $X$ and $G$ is a face of $F$, then $G$ is in $X$.
    \item If $F$ and $G$ are in $X$, then $F \cap G$ is a face of both $F$ and $G$.
\end{itemize}
The \emph{vertex set} $\VertSet(X)$ of $X$ is the set of 0-dimensional faces of $X$, and the \emph{facets} of $X$ are the faces which are maximal with respect to inclusion.
\end{defn}

\begin{defn}
The \emph{face poset} $\P_X$ of a cell complex $X$ is the subposet of $2^{\VertSet(X)}$ where each element of $\P_X$ consists of the set of vertices of some face $F$ of $X$.
\end{defn}

If $X$ is a cell complex, then the second property of \Cref{defn:polyhedralCellComplex} ensures that $\P_X$ is a meet-semilattice, with meet given by set intersection.
Note that if the facets of $X$ are simplices, then $\P_X$ is downward-closed, i.e.\ is a simplicial complex.

The face poset of a cell complex is interval Cohen-Macaulay, as follows from combining \Cref{prop:shellable} with the following lemma:

 \begin{lem}\cite[Thm 2.17(ii)]{ziegler_lectures_1995}\label{lem:faceposetInterval}
If $F \sbe G$ are two faces of a polytope $X$, then the interval $[F, G] \subseteq \P_X$ is the face poset of another polytope of dimension $\dim G - \dim F + 1$.
\end{lem}

As the (reduced) Euler characteristic of the boundary of a polytope is $\pm 1$, \Cref{thm:Betti2} can be rephrased in this context as follows.

\begin{cor}\label{cor:betticell}
For $F\leq G \in \mathcal P_X$ faces of a cell complex $X$, we have
$$\beta_{i,m(F,G)}(I^\star_{\C(\mathcal P_X)}) = \begin{cases} 1 & \textnormal{if $i = \dim G - \dim F$}\\ 0 & \textnormal{otherwise}.
\end{cases}
$$
\end{cor}

Since the rank of $\P_X$ is one more than the dimension of $X$, we get the first inequality of the following corollary.
\begin{cor}\label{cor:trinity2}
Let $\mathcal P_X$ be the face poset of a polyhedral cell complex $X$ of dimension $\dim X$, and $\C(\mathcal P_X)$ the associated function class.  Then
$$\vcdim \C(\P_X) \le \hdim \C(\P_X) = \dim X + 1.$$
Equality holds iff $X$ has a simplex of full dimension ($=\dim X$).
If any polytope in $X$ of maximal dimension has a simplex as a facet or a simple vertex (i.e. a vertex incident on exactly $\dim X$ edges), then
\[\hdim \C(\P_X) - 1 \le \vcdim \C(\P_X) \le \hdim \C(\P_X).\]
This is always the case if $\dim X \le 3$.

\end{cor}
\begin{proof}
As mentioned before, \Cref{cor:betticell} (or just \Cref{cor:pdim<rank}) shows the first inequality.

If $X$ has a $(\dim X)$-dimensional simplex, then the set of vertices of this simplex is shattered by the functions corresponding to the faces.
Conversely, $\vcdim \C(\P_X) = \hdim \C(\P_X)$ implies that there is a rank $(\dim X + 1)$ Boolean sublattice.
This sublattice, being maximal, must correspond to the face lattice of a maximal polyhedral cell. 
But any polytope with a Boolean face lattice has to be a simplex, so this yields the claim.

Similarly, if a $(\dim X)$-dimensional polytope in $X$ has a simplicial facet, then the $(\dim X)$ vertices of this facet are shattered, so that
\[
\dim X \le \vcdim \C(P_X).
\]
Likewise, when a $(\dim X)$-dimensional polytope has a simple vertex, then its $(\dim X)$ neighbor vertices are shattered by the faces of this polytope, and the same inequality holds.

Finally, we consider the case when $X$ has dimension at most 3.
Every edge shatters the two points it contains, so $\vcdim \C(\P_X) = \hdim \C(\P_X) = 2$ when $\dim X = 1$, and $2 \le \vcdim \C(\P_X) \le \hdim \C(\P_X) \le 3$ when $\dim X = 2$.

Now suppose $\dim X = 3$ (we may assume $X$ is a polytope).
We show that $X$ has to either have a triangular facet or a vertex with 3 neighbors.
Suppose not, and let $v, e, f$ respectively denote the number of vertices, edges, and 2-faces of $X$.
Since every vertex is incident on at least 4 edges, we have $2e \ge 4v$.
Moreover, since every face has at least 4 edges, we get $2e \ge 4f$.
But by Euler's formula, this means
\[
2 = v - e + f \le \frac 1 2 e - e + \frac 1 2 e = 0,
\]
a contradiction as desired.
\end{proof}

\begin{rem}
In dimension $\ge 4$, it is no longer true that every polytope has either a simplex facet or a simple vertex.
For instance, the 24-cell is a $4$-dimensional polytope in which every facet is an octahedron and each vertex is incident on 6 edges (cf. \cite{coxeter1973regular}).
\end{rem}

\begin{rem}
The assumption of having either a simple vertex or a simplex facet holds generically, in the sense that the convex hull of a set of points in general position is a simplicial polytope (i.e. has all facets being simplices), and the intersection of a collection of half-spaces in general position is a simple polytope (i.e. all of whose vertices are simple) \cite{ziegler_lectures_1995}.
\end{rem}

\section{Applications to computer science}\label{section:comp}

In this section, we apply our new tools to various function classes in computer science.
We first review some terminology: fix $d \in \mathbb{N}$, and set $n := 2^d$.

We consider function classes on $[n]$ consisting of Boolean formulas, as follows:
identify $[n]$ with the set $[2]^d$ of binary strings $(s_0s_1\ldots s_{d-1})$ of length $d$, and let $x_0, \ldots, x_{d-1}$ be Boolean variables, representing Boolean functions $x_i : [2]^d\to [2]$ that send $(s_0\ldots s_{d-1}) \mapsto 1$ if $s_i = 1$ and is 0 otherwise.  
The negation $\neg x_i: [2]^d \to [2]$ sends $(s_0\ldots s_{d-1}) \mapsto 1$ if $s_i = 0$ and is 0 otherwise.  
A \emph{literal} is a variable $x_i$ or its negation $\neg x_i$.
The \emph{conjunction} of a set of literals is their logical AND (denoted with $\wedge$).
The \emph{disjunction} of a set of literals is their logical OR (denoted with $\vee$).

A \emph{Boolean formula} is any expression that can be built up by conjunctions and disjunctions of literals, and naturally represents a function $[2]^d \to [2]$ mapping a binary string to its evaluation under the formula.

\begin{eg}
The conjunction of $\{x_1, \neg x_3, x_7\}$ is a Boolean function $x_1 \wedge \neg x_3 \wedge x_7: [2]^d \to [2]$ such that $(s_0\ldots s_{d-1}) \mapsto 1$ if $s_1 = 1$, $s_3 = 0$, and $s_7 = 1$, and is 0 otherwise.
Similarly, $x_1 \vee \neg x_3 \vee x_7$ is the Boolean function that is 0 iff $s_1=0, s_3=1$, and $s_7=0$. The expression $(x_1 \wedge \neg x_3) \vee (x_7 \wedge \neg x_1)$ is an example of a Boolean formula.
\end{eg}

\subsection{Application of results on polyhedral complexes}

The class of conjunctions arises as the indicator functions of the faces of a cube, where an empty conjunction (the constant function $1$) corresponds to the entire cube, the contradictory conjunction (e.g. $x_1\wedge \neg x_1$) corresponds to the empty set, and a nonrepeating conjunction of length $k$ (e.g. $x_1 \wedge \cdots \wedge x_k$) corresponds to a face of codimension $k$.
\Cref{cor:betticell} then recovers the Betti numbers of the class of conjunctions \cite[Section 2.3.4]{Yan17} by taking $X$ to be the cell complex of a cube $[0,1]^d$.

\subsection{Applications of the rank bound}\label{subsection:comp1}

In the following, we use our rank bound \Cref{cor:pdim<rank} to show that homological dimension of the class of $k$-CNFs is equal to its VC dimension, up to constant multiplicative factors.
We first recall the definition of $k$-CNF.
\begin{defn}
A \emph{$k$-CNF (Conjunctive Normal Form)} is a boolean formula that is a conjunction (AND) of a number of clauses
\[
C_1 \wedge \cdots \wedge C_m\quad
\text{(for example, when $k =2$,}\quad
(x_1 \vee \neg x_3) \wedge (x_2 \vee x_7) \wedge \neg x_2)
\]
where each clause $C_i$ is a disjunction (OR) of at most $k$ literals.
A \emph{monotone $k$-CNF} is a $k$-CNF without any negations appearing.
The class of (monotone) $k$-CNFs in $d$ variables is the class of functions in $[2]^d \to [2]$ consisting of functions corresponding to all (monotone) $k$-CNFs.
\end{defn}

\begin{thm}\label{thm:kCNF}
Let $\C$ be the class of $k$-CNFs and let $\C^+$ be the class of monotone $k$-CNFs in $d$ variables.
Then, with $e$ denoting Euler's constant,
\begin{align*}
    \binom d k \le \vcdim \C^+ &\le \hdim \C^+ \le \sum_{i=0}^k \binom{d} i \le (e d/k)^k\\
    \binom d k \le \vcdim \C &\le \hdim \C \le 2^k \binom{d} k
\end{align*}
so that
\begin{align*}
    \Omega(d^k) &\le \vcdim \C^+ \le \hdim \C^+ \le O(d^k)\\
    \Omega(d^k) &\le \vcdim \C \le \hdim \C \le O(d^k),
\end{align*}
where $\Omega$ and $O$ hide constants dependent on $k$ but independent of $d$.

\end{thm}
\begin{proof}
It was established in \cite{kearns_introduction_1994} that $\vcdim \C, \vcdim C^+ \ge \binom d k$ by noticing that the set of inputs
\[
\{x \in [2]^{d} \mid \sum_i x_i = d - k\}
\]
is shattered by $\C^+$ (and thus also by $\C$).
So it suffices to establish the upper bounds.
We start with the class of $k$-CNFs, and then deal with the monotone case.

\noindent\emph{$k$-CNF:}
We start with the upper bound of $\hdim \C$,
\begin{equation}\label{eqn:kCNF}
\hdim \C \le 2^k \binom{d} k.
\end{equation}
We prove this via the rank bound on homological dimension (\Cref{cor:pdim<rank}) and by showing that the rank of $\C$, as a poset in the natural partial order $f \le g \iff f^{-1}(1) \subseteq g^{-1}(1)$, is bounded by the right hand side of \eqref{eqn:kCNF}.

Consider a chain of functions $\boldsymbol 0 < f_1 < \cdots < f_m < \boldsymbol 1$ in $\C$, where $\boldsymbol 0$ (resp. $\boldsymbol{1}$) denotes the constant function 0 (resp. 1).
Each $f_i$ is a conjunction of disjunctive clauses,
\[
f_i = \bigwedge_j C_{ij},\quad \text{each $C_{ij}$ is a disjunction of at most $k$ literals.}
\]
Because $f_i < f_{i+1} < \cdots \le f_m$, we have
\[
f_i = \bigwedge_{i' = i}^m  f_{i'} = \bigwedge_{i' =i}^m \bigwedge_j C_{i'j}.
\]
Therefore we may assume that the clauses of the functions are in (strict) inclusion order
\[
\text{all clauses } \supset \{C_{1j}\}_j \supset \cdots \supset \mathcal \{C_{mj}\}_j \supset \emptyset.
\]
Furthermore, we can assume that the clauses all have exactly $k$ literals, as any size-$k'$ clause, $k' \le k$, can be written as a conjunction of such clauses.
For example,
\begin{align*}
x_1 \vee \cdots \vee x_{k'}
    &=
        x_1 \vee \cdots \vee x_{k'}
        \vee ( x_{k'+1} \wedge \neg x_{k'+1}) \vee \cdots 
        \vee ( x_{d} \wedge \neg x_{d})
        \\
    &=
        \bigwedge_{b \in [2]^{d-k'}}
            x_1 \vee \cdots \vee x_{k'} \vee 
            (\neg)^{b_1} x_{k'+1} \vee \cdots \vee (\neg)^{b_{d-k'}} x_d
\end{align*}
by the distributivity of $\wedge$ and $\vee$.
There are only $2^k \binom d k$ unique clauses with exactly $k$ literals (choose the $k$ variables first, and then decide whether to negate each of them).
Therefore, the chain above can be at most $2^k \binom d k$ long.

By \Cref{cor:pdim<rank}, this proves the desired upper bound on homological dimension.

\noindent\emph{Monotone $k$-CNF:}
The upper bound for $\C^+$ can be proved similarly, except here we cannot express a size-$k'$ clause, $k' \le k$, as a conjunction of size-$k$ monotone disjunctions.
The bound is then established by noting that there are $\sum_{i=0}^k \binom d i$ unique clauses of size $\le k$.
\end{proof}

\begin{rem}
When $k = 1$, the class of (resp.\ monotone) $k$-CNFs is just the class of (resp.\ monotone) conjunctions.
According to \cite[Section 3.1]{Yan17}, the homological dimension of (resp.\ monotone) conjunctions in $d$ Boolean variables is $d+1$ (resp.\ $d$).
At the same time, \Cref{thm:kCNF} only says that the homological dimension is between $d$ and $2 d$ (resp.\ $d$ and $d+1$), so the upper bound of \Cref{thm:kCNF} is not tight in this case.
\end{rem}
\begin{rem}
The logic of \Cref{thm:kCNF} can be applied straightforwardly to bound the homological dimension of \emph{CSP classes,} which we discuss now.
\end{rem}

In general, given a collection of Boolean functions $f_i: [2]^d \to [2]$, their \emph{conjunction} $\bigwedge_i f_i$ is the function that sends $v \in [2]^d$ to 1 iff $f_i(v) = 1$ for all $i$.
Likewise, their \emph{disjunction} $\bigvee_i f_i$ is the function that sends $v \in [2]^d$ to 0 iff $f_i(v) = 0$ for all $i$.
These definitions generalize the notions of conjunction and disjunction introduced earlier for literals.

\begin{defn}
Let $\D$ be a set of Boolean functions on $[2]^d$.
The \emph{class of $\D$-CSPs (Constraint Satisfaction Problems)} is the conjunction closure of $\D$, i.e. it contains all functions of the form $\bigwedge_i f_i$ where each $f_i \in \D$.
\end{defn}
For example, if we let $\D$ be the set of all (resp. monotone) disjunctions of size at most $k$, then $\D$-CSPs are just (resp. monotone) $k$-CNFs.
\begin{thm}
The class $\C$ of $\D$-CSPs satisfies
\[
\vcdim \C \le \hdim \C \le |\D|.
\]
\end{thm}
\begin{proof}

We consider the natural semilattice structure of $\D$ induced from the conjunction closure of $\C$.
By the same reasoning as in the proof of \Cref{thm:kCNF}, any chain in this semilattice $\boldsymbol{0} < f_1 < \cdots f_m < \boldsymbol{1}$ must correspond to a chain of reverse inclusions
\[
\D \supset \{g_{1j}\}_k \supset \cdots \supset \{g_{mj}\}_j \supset \emptyset
\]
of collections of functions $g_{ij} \in \D$, in such a way that $f_i = \bigwedge_j g_{ij}$.
This chain can be at most $|\D|$ long since each strict inclusion must differ by some new function in $\D$.
By \Cref{cor:pdim<rank}, this yields the upper bound on homological dimension.
\end{proof}

Note that this last bound is in general far from sharp: it follows from \Cref{cor:parityconj} below that the class of parity functions in $d$ Boolean variables and its conjunction closure have the same VC-dimension $d$, but the size of the class of parity functions is $2^d$.

\subsection{Applications of results on matroids} \label{subsection:comp2}
We conclude with some applications of \Cref{cor:trinity1}.

\begin{cor}
\label{cor:parityconj}
For the class $\C$ of conjunctions of parity functions in $d$ variables,
\[
\vcdim \C = \hdim \C = d.
\]
\end{cor}
\begin{proof}
Consider the representable rank-$d$ matroid given by all vectors in $\FF_2^d$.
The flats of this matroid are the subspaces of $\FF_2^d$, and the rank function is the vector space dimension over $\FF_2$.

Recall that a parity function $[2]^d \to [2]$ is an $\FF_2$-linear functional by identifying $[2] \cong \FF_2$.
Every parity function is the indicator function of a hyperplane in $\FF_2^d$, and every conjunction of parity functions is the indicator function of a subspace of $\FF_2^d$, which is an intersection of hyperplanes.
Thus the class of conjunctions of parity functions is exactly the function class associated to the matroid above.
\Cref{cor:trinity1} then yields the result.
\end{proof}

By considering suitable (squarefree) symmetric powers -- analogous to taking a Veronese embedding in algebraic geometry -- we can generalize \Cref{cor:parityconj} to higher degree polynomials over $\FF_2$.
\begin{cor}
\label{cor:polyconj}
For the class $\C \sbe [2]^{[2]^d}$ of conjunctions of polynomials over $\FF_2$ of degree $\le k$,
\[
\vcdim \C = \hdim \C = \sum_{i=0}^k \binom d i.
\]
\end{cor}
\begin{proof}
Let $D := \sum_{i=0}^k \binom d i$, and let $\FF_2^D$ have coordinates $\{y_U \mid U \subseteq [d], |U| \le k \}$.
Consider the embedding
\[
\phi: \FF_2^d \to \FF_2^D, \;
x=(x_1, \ldots, x_d) \mapsto \left( \prod_{i \in U} x_i \right)_{U \sbe [d], |U| \le k}
\]
sending $(x_i)$ to the vector of all monomials in $x_i$ with degree $\le k$.

The linear matroid given by the image of $\phi$ has rank $D$.
Any flat of this matroid corresponds to the zero set of a system of linear equations in the $y_U$, hence via pullback by $\phi$, to a system of polynomial equations $\{p_i(x)\}_i$ of degree at most $k$.

The indicator function of such a set is the conjunction of the Boolean functions $\{p_i(x)-1\}_i$:
\[
p_i(x) = 0 \quad \forall i \iff
\text{$x$ satisfies }\bigwedge_i (p_i(x)-1).
\]
Thus the class $\C$ under consideration is exactly the class of conjunctions of $\FF_2$-polynomials with degree at most $k$, and the claim follows from \Cref{cor:trinity1}.
\end{proof}

\begin{rem}\label{rem:conjunctionNoHarm}
Since the VC dimension and homological dimension of parity functions (resp. $\FF_2$-polynomials with degree at most $k$) are both $d$ (resp. $\sum_{i=0}^k \binom d i$) as well \cite[Section 3.1]{Yan17}, the above results show that adding the operation of conjunction does not increase the ``complexity'' of these classes, from both a learning-theoretic and a homological point of view.
\end{rem}

\bibliography{Biblio}
\bibliographystyle{alpha}

\end{document}